\documentclass{amsart}
\usepackage{graphicx}
\usepackage{indentfirst}
\usepackage{hyperref}

\usepackage{latexsym}
\usepackage{amsmath}
\usepackage{amsthm}
\usepackage{amssymb}
\usepackage{graphicx}
\usepackage{pdfpages}
\usepackage[labelformat=empty]{caption}
\usepackage{CJKutf8}

\newtheorem{theorem}{Theorem}[section]
\newtheorem{lemma}[theorem]{Lemma}

\theoremstyle{definition}

\newtheorem{prop}[theorem]{Proposition}
\theoremstyle{remark}

\newcommand{\vip}{\vskip.2cm}
\newcommand{\R}{\mathbb{R}}

\newcommand{\E}{\mathbb{E}}
\newcommand{\p}{{\mathbb{P}}}

\newcommand{\indiq}{{\bf 1}}

\newcommand{\clim}{\lim_{T\to\infty}}
\newcommand{\cint}{\int_0^\infty}

\hypersetup{pdfmenubar=true}

\begin{document}
\begin{CJK}{UTF8}{gbsn}

\title[Limit theorem for unstable Hawkes processes]{SCALING LIMITS FOR SUPERCRITICAL NEARLY UNSTABLE  HAWKES PROCESSES}

\author[C. Liu]{Chenguang Liu}
\author[L. Xu]{Liping Xu}
\author[A.Zhang]{An Zhang}

\address{Delft Institute of Applied Mathematics, EEMCS, TU Delft, 2628 Delft, The Netherlands}
\email{C.Liu-13@tudelft.nl}
\address{School of mathematical sciences, Beihang University, PR China}
\email{xuliping.p6@gmail.com}
\address{School of mathematical sciences, Beihang University, PR China}
\email{anzhang@buaa.edu.cn}



\begin{abstract} 
In this paper, we investigate the asymptotic behavior of nearly unstable Hawkes processes whose regression kernel has $L^1$ norm strictly greater than one and close to one as time goes to infinity.  We find that, the scaling size determines the scaling behavior of  the processes like in \cite{MR3313750}. Specifically, after suitable rescaling, the limit of the sequence of Hawkes processes is deterministic. And also with another appropriate rescaling, the  sequence converges in law to an integrated Cox–Ingersoll–Ross like process. This theoretical result may apply to model the recent COVID-19 in epidemiology and in social network.

\end{abstract}

\maketitle

\section{Introduction and main results}
\subsection{Introduction}
Hawkes process $(Z_t)_{t\ge0}$ is a random point process   that describes
temporal stochastic self-exciting phenomena which is first introduced by Hawkes \cite{MR278410} in 1971.  The evolution of the process is influenced by the timing of the past events. It is determined by its intensity process $(\lambda_t)_{t\ge0}$, which is of the form
\[\lambda_t=\mu+\int_0^t\varphi(t-s)dZ_s.\]
where $\mu>0$ is interpreted as  the ``background rate'' or ``exogenous rate" of the process,  and $\varphi$ is the regression kernel,  which is the intensity function of a nonhomogeneous Poisson process, and quantifies the influence of past  events of the process on the arrival of future events.  More introduction related to Hawkes processes,  see \cite{MR1950431,MR2371524}. 
\vip
Hawkes processes not only is  an  interesting subject in mathematics but also  have been widely applied to various fields to model both natural and social phenomena, like originally modeling earthquakes and their aftershocks in seismology \cite{MR278410,MR378093}，brain activity in neuroscience \cite{MR3322314}, risk estimation,  transaction times and midquote changes and so on in mathematical finance \cite{MR3219705,MR3480107,MR3750729} and social networks interactions \cite{SIR2018,Social}, etc. There are substantial numbers of works on Hawkes processes and its applications in various disciplines, here we give a  nonexhaustive reference   and  we refer the reader to the bibliography  and references therein for more details.  
\vip
From a mathematical  point of view, it's very nature  to investigate  the large time behavior of Hawkes process, especially the Law of Large Numbers and the Central Limit Theorem, which have been established for different cases of  Hawkes processes, like linear, nonlinear, subcritical,  critical and supercritical, under the fixed rate assumption $\cint\varphi dt<\infty$, see \cite{MR3102513,MR3395721,MR4418230,MR3054533,MR3187500}. Recently,  the limit theorem for the multivariate marked Hawkes processes was established,  see \cite{MR4757489}. Meanwhile,  the functional central limit theorems for subcritical and critical Hawkes processes and the convergence for the heavy-tailed Hawkes processes were investigated as well in \cite{Xu2024} and \cite{Horst2023}, respectively. In particular, for the linear regime, because of its tractability especially the immigration-birth representation, the limit behaviors have been well understood and widely applied in practice. We  also restrict our consideration to the linear case.  And the rate $\cint\varphi dt<1$ refers to the subcritical case,  $\cint\varphi dt=1$ refers to the critical  case and otherwise the supercritical case.  In finance,  this fixed rate is the so-called {\it branching  ratio}, which is interpreted as the average proportion of endogenously triggered events. However, in practice, this ratio can't be always fixed, it may vary from time to time caused by many real factors. This is consistent with the statistical estimation results, which seems to be shown very often, only nearly unstable Hawkes processes are able to fit the data properly, see the introduction of \cite{MR3313750}.  Regarding to the nearly unstable, first introduced by  Jaisson and Rosenbaum in \cite{MR3313750}, means that the $L^1$ norm of their kernel is close to one. In \cite{MR3313750},  Jaisson and Rosenbaum assume a time depending kernel  with its $L^1$ norm less than $1$ and close to $1$ as time goes to infinity, which is a  near instability condition, i.e.  
\[\cint \phi^Tdt=a_T<1, \  \hbox{and} \lim_{T\to\infty}a_T=1,\]
where $(a_T )_{T\ge0}$ is a sequence of positive numbers such that for all $T$, $a_T<1$. Together with the assumption $\cint \phi^Ttdt=m<\infty$, they then show that, these nearly unstable Hawkes process,  under  suitable  scaling, converges  deterministically   if $T(1-a_T)\to+\infty$  as $T\to\infty$, that is, (Law of Large Numbers)
\[\sup_{t\in[0,1]}\frac{1-a_T}{T}|Z_{Tt}^T-\E[Z^T_{Tt}]|\to 0   \ \hbox{in probability},\]
and tends to an integrated Cox-Ingersoll-Ross process if $T(1-a_T)\to\lambda>0$  as $T\to\infty$, that is, (Central Limit Theorem)
\[\Big( \frac{1-a_T}{T}Z^T_{Tt} \Big)_{t\in[0,1]}\to\Big( \int_0^t X_s ds \Big)_{t\in[0,1]} \  \hbox{in law} \]
for the Skorohod topology, where $(X_t)_{t\in[0,1]}$ is a  Cox-Ingersoll-Ross process, which is  originally introduced  in \cite{MR785475} and classically used to model stochastic volatilities in finance. 

\vip

\subsection{Motivation and comparison.}

As we noticed, the great bulk of literature related to the
Hawkes process and its applications to geophysics and finance focuses
on the subcritical regime, whereas the supercritical regime is little studied but is  more germane to epidemics, like modeling the spread of COVID-19 in \cite{MR4276450}, and to the social network \cite{Social, SIR2018}. 

\vip

In this paper, we consider the supercritical nearly unstable Hawkes processes, i.e. the kernels of the Hawkes processes depending on $T$, with $L^1$ norm strictly larger than one and close to one  as $T\to\infty$. We prove that under the assumption that the norm of kernel close to one at a  suitable speed (see the assumption $H(\lambda)$),  the limit behavior of our sequence of Hawkes processes, with a suitable rescaling, is an integrated diffusion process, like a CIR process, which is a strong solution of a Brownian driven SDE. To do so, we adopt a different approach from that in \cite{MR3313750}, where in order to get the limiting behavior of our sequence of Hawkes processes, Jaisson and Rosenbaum studied the limiting behavior of the intensity of the processes. Roughly speaking,  for $t\in[0,1]$, they rewrote the intensity $\lambda^T_{Tt}$ to $C_t^T=\lambda^T_{Tt}(1-a_T)$, which can be written as a stochastic integral equation, then prove the tightness of $C_t^T$ and pass the coefficients of the equation to the limit to obtain the limiting law. However, in our context, instead of considering the intensity, we directly deal with the Hawkes processes, following the spirit of  \cite[Chapter 5.3]{MR3187500}. Finally, Let's also mention the mean field limit of Hawkes processes studied in \cite{MR3499526,MR4127342}.  

\subsection{Setting}
We consider a  measurable  function $\varphi:[0,\infty)\to [0,\infty)$  and  a  Poisson measure $(\Pi(dt,dz))$,  on $[0,\infty)\times [0,\infty)$ with intensity $dtdz$. The sequence $\{a_T\}_{T\ge 0}$ indexed by $T$ is positive and $\lim_{T\to \infty} a_T=1.$ We put $\varphi^T=a_T\varphi.$
We consider the following system indexed by $T$:  all $t\geq 0,$
\begin{align}\label{sssy}
Z_{t}^{T}:=\int^{t}_{0}\int^{\infty}_{0}\boldsymbol{1}_{\{z\le\lambda_{s}^{T}\}}\Pi(ds,dz), \hbox{ where }
\lambda_{t}^{T}:=\mu+\int_{0}^{t-}\varphi^T(t-s)dZ_{s}^{T}.
\end{align}
In this paper, $\int_{0}^{t}$ means $\int_{[0,t]}$, and $\int_{0}^{t-}$ means $\int_{[0,t)}$. The solution $(Z_{t}^T)_{t\ge 0}$ is a  counting processes. 
By \cite[Proposition 1]{MR3499526}, see also \cite{MR1411506,MR3449317}, the system $(1)$ has a unique $(\mathcal{F}_{t})_{t\ge 0}$-measurable c\`adl\`ag 
solution, where   
$$\mathcal{F}^T_{t}=\sigma(\Pi^{i}(A):A\in\mathcal{B}([0,t]\times [0,\infty)),$$ as soon as $\varphi^T$ is locally integrable.

\subsection{An illustrating example in epidemiology}
Let's first recall the noticeable branching representation of Hawkes processes  in terms of
Galton-Watson trees given by Oakes-Hawkes \cite{MR378093}. That is, in a population process , the migrants arrive according to a Poisson process with  ``exogenous rate"  $\mu$. Then each migrant gives birth to children according to ``endogenous rate" given by $\varphi$, and each children gives birth to grandchildren according to the same ``endogenous rate". 
\vip
Now, let us consider the epidemical example. Let process $(Z_{t}^{T})_{t\ge 0}$ represent the number of infected individuals over the observation period $[0, T]$. This process accounts for two main mechanisms of infection: $\textbf{immigration}$ and $\textbf{transmission}$ from already infected individuals.
\vip
1. $\textbf{Immigration of Infected Individuals}:$ New infections enter the population independently of current local transmissions, modeled by a homogeneous Poisson process with rate $\mu$. This represents external sources of infection, such as travelers or new arrivals carrying the disease. 
\vip

2.  $\textbf{Local Transmission Dynamics}:$ Once an individual is infected, they can further spread the infection to others. This local transmission is captured by an inhomogeneous Poisson process, where the intensity is given by the kernel function $a_T\varphi(t-s).$ Here:

$\bullet$ $t$ denotes the current time,

$\bullet$  $s$ is the time at which an individual was infected,

$\bullet$ $\varphi$ is a kernel function, typically assumed to be decreasing. This assumption reflects the reality that the infectivity of an individual often decreases over time due to factors such as recovery or isolation.

\vip

The parameter $a_T$ represents the level of social interference, which encompasses various control measures such as medical interventions, government policies, and vaccination efforts. Notably, $a_T$ is a function of time $T$, and it typically decreases as $T$ increases, reflecting the effectiveness of these interventions in reducing the spread of infection over time. So, from this point of view, $a_T>1$ is quite consistent with the reality.

\subsection{Assumptions}
We always assume  
\renewcommand\theequation{{$\Lambda$}}
\begin{equation}
  \|\varphi\|_1= \int_0^\infty \varphi(s)ds=1\  \hbox{and}\ \int_0^\infty s\varphi(s)ds=m>0.
\end{equation}
And  for $\lambda> 0$,  we put the assumption:
\renewcommand\theequation{{$H(\lambda)$}}
\begin{equation}
    a_T> 1\  \hbox{and}\ \lim_{T\to\infty}T(a_T-1)=\lambda,
\end{equation}
or 
\renewcommand\theequation{{$H(\infty)$}}
\begin{equation}
    a_T> 1\  \hbox{and}\ \lim_{T\to\infty}T(a_T-1)=\infty.
\end{equation}
\renewcommand\theequation{\arabic{equation}}\addtocounter{equation}{-2}
The assumption $\cint s\varphi(s)ds=m>0$ is crucial in our context which enables us to approximate $\Psi^T$ by an exponential function in Lemma \ref{phito}.  
\subsection{Notation}
Let's first recall the convolution of two functions $f,g: [0,\infty)\mapsto \R$, which is defined by (If it exists) $(f\star g)(t)=\int_0^tf(t-s)g(s)ds$. $(\varphi)^{\star n}$ represents the $n$-th convolution product of function $\varphi$ and $(\varphi)^{*1}=\varphi$.
Since $\int_0^\infty \varphi(s)ds=1$, it's not hard to observe that  
$\cint (\varphi^T)^{\star n}(s)ds=(a_T)^n$. We adopt the conventions $\varphi^{\star0}(s)ds=\delta_0(ds)$ and  $\varphi^{\star0}(t-s)ds=\delta_t(ds)$, and  the convention $\varphi^{\star n}(s)=0$ for $s<0$. We also recall the $L^1$ norm of function $\varphi$ denoted by  $\|\varphi\|_1= \int_0^\infty \varphi(s)ds$.

\vip
Since $ 1<a_T=\int_0^\infty \varphi^T(s)ds<+\infty,$ we can choose a  unique  positive sequence $\{b_T\}_{T\ge 0}$ such that
\begin{align*}
    \int_0^\infty e^{-b_T s}\varphi^T(s)ds=\frac{1}{a_T}.
\end{align*}
$b_T$ is  sometimes referred to  the Malthusian parameter in the literature for fixed $T$, see \cite{MR3187500}.
We define the following functions on $\R^+$,
\begin{align}\label{phipsi}
    \Tilde \varphi^T(t)=e^{-b_T t}\varphi^T(t)\ \hbox{and}\  \Tilde \Psi^T(t)=\sum_{n\ge 1}(\Tilde \varphi^T)^{*n}(t).
\end{align}
Note that $\Tilde \Psi^T(t)$  is well defined on $\R^+$ since $\| \Tilde \varphi^T\|_1=1/a_T<1$.  
A direct computation implies that 
\begin{align}\label{norm}
\|\Tilde \Psi^T\|_1=\frac{\| \Tilde \varphi^T\|_1}{1-\| \Tilde \varphi^T\|_1}=\frac{1}{a_T-1}.
\end{align}
We also define at $t\ge0$,
\begin{equation}\label{PSI}
\Psi^T(t):= e^{b_T t}   \Tilde \Psi^T(t),
\end{equation}
which is well defined as well.
For $n\ge 1$, It's not hard to find that $e^{-b_T t}( \varphi^T)^{*n}(t)=(e^{-b_T t} \varphi^T)^{*n}(t)=(\Tilde \varphi^T)^{*n}(t).$ 
Thus, we have 
\begin{equation}\label{PSI2}
\Psi^T(t)=\sum_{n\ge 1}(\varphi^T)^{*n}(t).
\end{equation}
\vip
Let's recall $(\lambda^T_t)_{t\ge0}$ defined in \eqref{sssy},  and  introduce  the compensated Poisson measure $\Tilde \Pi(ds, dz)=\Pi(ds, dz)-dsdz$ associated to the Poisson measure $\Pi$. For $t\ge0$, set 
\begin{align}\label{mart}
M_t^T:=\int_0^t\cint \indiq_{\{z\le \lambda_s^T\}}\Tilde \Pi(ds, dz),
\end{align}
which is a martingale process related to $Z^T_t$.   We  then readily find that 
\[M_t^T=Z_{t}^{T}-\int_0^t \lambda^T_s ds.
\]
According to \cite[Remark 10]{MR3499526},  refer also to \cite[Chapter 1, Section 4e]{MR1943877}  for definitions and properties of pure
jump martingales and of their quadratic variations, we then have $ \E[M^{T}_s M^{T}_t]= \E[Z^{T}_{s\land t}]$ and  $[M^T,M^T]_t=Z_t^T$. 

\subsection{Main results}
We first give the law of large numbers in the following sense.
\begin{theorem}\label{thm1}
Under the assumption $(\Lambda)$ and $H(\infty)$, the
sequence of Hawkes processes is asymptotically deterministic, in the sense that
the following convergence in $L^2$ holds as $T\to \infty$,
\begin{align*}
    \sup_{x\in[0,1]}\frac{a_T-1}{Te^{b_TTx}}\Big|Z_{Tx}^{T}-\E[Z_{Tx}^{T}]\Big|\to 0.
\end{align*}
\end{theorem}

Next, we introduce our second main result related to the  central limit theorem.

\begin{theorem}\label{thin}
Under the assumption $(\Lambda)$ and $H(\lambda)$,
\begin{align*}
   \Big( \frac{Z_t^T}{T^2}\Big)_{t\in  [0,1]}\stackrel{(d)}\longrightarrow \Big(\int^t_0X_sds\Big)_{t\in [0,1]} \quad \hbox{as  $T\to\infty$,}
\end{align*}
where $(X_t)_{t\in[0,1]}$ is  the  unique strong solution to the following SDE
\begin{align*}
     X_t=\frac{1}{m}\int_0^t\Big(\mu+\lambda X_s\Big)ds+\frac{1}{m}\int_0^t \sqrt{X_s}dB_s,
\end{align*}
where $(B_t)_{t\ge0}$ is a Brownian motion.
\end{theorem}

\subsection{Arrangement of the paper}  Section 2, we study the limit behavior of the important convolution function and of the expectation of the scaled Hawkes processes.  Finally, we prove the limit behavior of the  Hawkes processes.
\vip
In the sequel, $C$ stands for a positive constant whose value may change from line to line. When necessary, we will indicate in subscript the parameters it depends on.

\section{Analysis of the convolution function}
The aim of this section is to provide some auxiliary results on the kernel function $\varphi.$ 
 We are going to illustrate the long-term behavior of the function $\Psi^T$.  The first lemma could be seen as a corollary of \cite[Proposition 2.2]{MR3313750}.
 
 \begin{lemma}\label{lem: cT}
     Let $\{c_T\}_{T\ge 0}$ is a positive sequence with $0< c_T<1$.  Then we have the following convergence for $x\ge 0$ in the weak sense:
\begin{align*}
    \lim_{T\to \infty}\hat \Psi^T(Tx)=&\frac{1}{m}e^{-\frac{\lambda}{m}x},\ \  \textit{if $\lim_{T\to\infty}T(1-c_T)=\lambda>0$  }\\
    \lim_{T\to \infty}\hat \Psi^T(Tx)=&0,\ \textit{ if $\lim_{T\to\infty}T(1-c_T)=+\infty$},
\end{align*}
where $\hat  \Psi^T(t)= \sum_{k=1}^\infty (c_T)^k(\varphi)^{*k}(t).$
 \end{lemma}
 
\begin{proof}
We define a random variable
\begin{align*}
    X^T:= \frac{1}{T} \sum_{i=1}^{I^T} X_i,
\end{align*}
where $(X_i)$ are i.i.d. random variables with density $\varphi$ and $I^T$ is a geometric
random variable independent of $(X_i)$ with parameter $1-c_T$. We denote the density function of $X^T$ by $\rho^T(x)$, which is of the form, for $x\ge0$, $$\rho^T(x)=T\frac{\hat  \Psi^T(Tx)}{\|\hat  \Psi^T\|_1}.$$ 
Indeed,  recalling $\hat  \Psi^T(t)= \sum_{k=1}^\infty (c_T)^k(\varphi)^{*k}(t),$ we can easily compute that $\|\hat  \Psi^T\|_1=\frac{c_T}{1-c_T}$ and have
\begin{align*}
\p(X^T\le x)&=\sum_{k=1}^\infty\p\left(\sum_{i=1}^k X_i\le Tx\right)\p(I^T=k)\\
&=\sum_{k=1}^\infty \int_0^{Tx} \varphi^{*k}(u)du(1-c_T)c_T^{k-1}\\
&=\frac{1-c_T}{c_T}\int_0^{Tx} \hat \Psi^T(u)du.
\end{align*}
Thus, we have
\begin{align*}
\hat  \Psi^T(Tx)=\rho^T(x)\frac{\|\hat  \Psi^T\|_1}{T}=\rho^T(x)\frac{c_T}{(1-c_T)T}.
\end{align*}
In order to compute the limit of $\hat  \Psi^T(Tx)$, it suffices to compute the limit of $\rho^T(x)$.
We thus look at the characteristic function of $X^T$, which is the Fourier transform of $\rho^T(x)$, denoted by $\hat{\rho}^T$, satisfying
\begin{align*}
\hat{\rho}^T(z)&=\E[e^{izX^T}]\\
&=\sum_{k=1}^\infty\E[e^{i(z/T)\sum_{i=1}^kX_i}]P(I^T=k)\\
&=\sum_{k=1}^\infty(1-c_T)(c_T)^{k-1}\E[e^{i(z/T)\sum_{i=1}^kX_i}]\\
&=\sum_{k=1}^\infty(1-c_T)(c_T)^{k-1}\Big(\hat\varphi(z/T)\Big)^{k}\\
&=\frac{\hat\varphi(z/T)}{1-(c_T/(1-c_T))(\hat\varphi(z/T)-1)},
\end{align*}
where $\hat\varphi$ is the characteristic function of $X_1$. The function $\hat\varphi$ is continuously differentiable  with that $\hat\varphi'(0)=i\E(X_1)=im$ because $\varphi$ is the density of $X_1$ and that $\E(X_1)=\int_0^\infty s\varphi(s)ds=m>0$. Therefore, using that $\lim_{T\to\infty}c_T=\lim_{T\to\infty}\hat\varphi(z/T)=\hat\varphi(0)=1$, we have 
\begin{align*}
\lim_{T\to\infty}\hat{\rho}^T(z)&=\lim_{T\to\infty}\frac{\hat\varphi(z/T)}{1-(c_T/(1-c_T))(\hat\varphi(z/T)-1)}\\
&=\lim_{T\to\infty}\frac{\hat\varphi(z/T)}{1-(zc_T/T(1-c_T))(\hat\varphi(z/T)-1)/(z/T)}\\
&=\lim_{T\to\infty}\frac{1}{1-imz/T(1-c_T)}.
\end{align*}
If $\lim_{T\to\infty}T(1-c_T)=\lambda> 0$, then 
\[\lim_{T\to\infty}\hat{\rho}^T(z)=\frac{1}{1-i(m/\lambda)z},\]
the  right hand side  is nothing but the characteristic function of the exponential distribution  with parameter $\lambda/m$. Hence,  by the continuity theorem, we conclude that $X^T\stackrel{(d)}\to X$ with $X\sim \mathcal{E}(\lambda/m)$ so that 
\[\lim_{T\to\infty}\hat  \Psi^T(Tx)=\frac{1}{\lambda} \frac{\lambda}{m}e^{-\frac{\lambda}{m}x}=\frac{1}{m}e^{-\frac{\lambda}{m}x}.\]
If $\lim_{T\to\infty}T(1-c_T)=+\infty$, then 
\[\lim_{T\to\infty}\hat{\rho}^T(z)=1,\] 
which means that $X^T$ converges in law to $0$.  We thus conclude $\lim_{T\to\infty}\hat  \Psi^T(Tx)=0$.
 
\end{proof}

In the next lemma, we will see that $\Psi^T$ with the sequence $a_T>1$ has the same property as Lemma \ref{lem: cT} of $\hat \Psi^T$ with $0<c_T<1$. 
\begin{lemma}\label{phito}

 Under the assumption $(\Lambda)$ and $H(\lambda)$, recall $\Psi^T$ defined in \eqref{PSI}, we have the following convergence in the weak sense  for $x\ge 0$:
 $$\lim_{T\to \infty}\Psi^T(Tx)=\frac{1}{m}e^{\frac{\lambda}{m} x}.$$
 
\end{lemma}

\begin{proof}
Let's first recall  that the positive sequence $\{b_T\}_{T\ge 0}$ satisfy
\begin{align*}
    \int_0^\infty e^{-b_T s}\varphi^T(s)ds=\frac{1}{a_T}
\end{align*}
and  define 
\begin{align*}
    m_T:=\int_0^\infty e^{-b_T s}s\varphi^T(s)ds.
\end{align*}
Since $\clim a_T=1,$  then $\{b_T\}_{T\ge 0}$ must satisfy $\clim b_T=0.$  Hence, by the dominated convergence theorem, we have  $\clim m_T=m$  and then $\clim \Big(a_T-\frac{1}{a_T}\Big)/b_T=m.$ Indeed,
\begin{align*}
  \clim \Big(a_T-\frac{1}{a_T}\Big)/b_T&=\clim a_T\int_0^\infty(1- e^{-b_T s})\varphi(s)/b_Tds\\
  &=\int_0^\infty\clim(1- e^{-b_T s})\varphi(s)/b_Tds\\
  &=\cint s\varphi(s) ds=m.
\end{align*}
Due to $\clim T(a_T-1)=\lambda,$  we thus conclude that $\clim Tb_T=2\lambda/m.$
Note that $\cint s\Tilde \varphi^T(s) ds=m_T$ and  recall \eqref{phipsi} that 
 $$
 \Tilde \varphi^T(t)=e^{-b_T t}\varphi^T(t)\ \hbox{and}\  \Tilde \Psi^T(t)=\sum_{n\ge 1}(\Tilde \varphi^T)^{*n}(t).
 $$
  As a result of Lemma  \ref{lem: cT}, we have 
\begin{align*}
    \lim_{T\to \infty}\Tilde \Psi^T(Tx)=\frac{1}{m}e^{- \frac{\lambda x}{m}}.
\end{align*}
Consequently, according to $\Psi^T(t)= e^{b_T t}   \Tilde \Psi^T(t)$ defined in \eqref{PSI}, 
$$
\clim \Psi^T(Tx)= \clim  e^{Tb_T x}\Tilde \Psi^T(Tx)=e^{\frac{2\lambda}{m} x}\frac{1}{m}e^{- \frac{\lambda x}{m}}=\frac{1}{m}e^{\frac{\lambda}{m} x},
$$
which accomplishes the proof.

\end{proof}

We now recall a classical lemma,  see \cite[Lemma 3]{MR3054533} and \cite[Lemma 8]{MR3499526}. For the reader's convenience, we will sketch the proof. We also review the convention that  $\varphi^{*0}(t-s)ds=\delta_t(ds)$. 
\begin{lemma}\label{classic}
Let $A$ be a constant. Consider $f, g: [0,\infty)\mapsto\R$ locally bounded and assume that $\varphi: [0,\infty)\mapsto[0,\infty)$ is locally integrable. If $f_t=g_t+\int_0^t \varphi(t-s)Af_s ds$ for all $t\ge0$,
then 
\[f_t=\sum_{n\ge0}\int_0^t \varphi^{*n}(t-s)A^ng_sds.\]
\end{lemma}

\begin{proof} For unknown $f$, we prove that the equation $f_t=g_t+\int_0^t \varphi(t-s)Af_s ds$ has a unique locally bounded solution. Given $f,\tilde f$ are two solutions,  then we find that $h=|f-\tilde f|$ admits $h_t\le |A|\int_0^t \varphi(t-s) h_s ds$, by the generalized Gronwall lemma, see \cite[Lemma 23-(i)]{MR3449317},  we deduce that $h=0$. We thus need to verify that $f_t=\sum_{n\ge0}\int_0^t \varphi^{*n}(t-s)A^ng_sds$ is locally bounded and solves the equation $f_t=g_t+\int_0^t \varphi(t-s)Af_s ds$. It's easy to check that $f_t=\sum_{n\ge0}\int_0^t \varphi^{*n}(t-s)A^ng_sds$ solves the equation. In fact, using the property of convolution, we write
\begin{align*}
f_t&=g_t+\sum_{n\ge1}A^nc\star g(t)\\
&=g_t+\sum_{n\ge1} A^{n-1}\varphi^{*n-1}\star(A\varphi)\star g(t)\\
&=g_t+(A\varphi)\star \sum_{n\ge0} A^{n}\varphi^{*n}\star g(t)=g_t+(A\varphi)\star f(t).
\end{align*}
Finally, set $u_t^n=|A|^n \int_0^t\varphi^{*n}(s)ds$, which is locally bounded because
$\varphi$ is locally integrable and which satisfies $u_t^{n+1}\le |A|\int_0^t u_s^n\varphi(t-s)ds$. By \cite[Lemma 23-(ii)]{MR3449317}, we conclude that $\sum_{n=0}^\infty u_t^n$ is locally bounded. Whence, $f_t$ is locally bounded because $g$ is locally bounded.
\end{proof}

Next, we are trying to estimate the expectation of the Hawkes process, which is crucial to prove the tightness.

\begin{prop}\label{estiZ}
Consider the solution  $(Z_t^T)_{t\ge0}$ to \eqref{sssy}.
\vip
 (i) Under the assumption $(\Lambda)$ and $H(\lambda)$, we have for  $0\le x\le 1$
\begin{align*}
    \clim \E[Z_{Tx}^{T}]/T^2&=\frac{\mu e^{\frac{\lambda}{m}x}}{m}\int_0^x ve^{-\frac{\lambda}{m}v}dv =-\frac{\mu}{\lambda}+\frac{m\mu}{\lambda^2}\Big(e^{\frac{\lambda}{m}x}-1\Big).
\end{align*}

\vip
(ii) Under the assumption $(\Lambda)$ and $H(\infty)$,  we have
\begin{align*}
    \sup_{x\in[0,1]}\frac{a_T-1}{Te^{b_TTx}}\E[Z_{Tx}^{T}]\le \mu  a_T.
\end{align*}
Especially,
\begin{align*}
    \sup_{t\in[0,T]}\frac{a_T-1}{Te^{b_Tt}}\E[Z_{t}^{T}]\le \mu  a_T.
\end{align*}
\end{prop}

\begin{proof} The proof is inspired by  \cite[Lemma 11]{MR3499526}, which deals with the multivariate Hawkes processes. We first rewrite the Hawkes process $(Z_t^T)_{t\ge0}$. Starting from \eqref{sssy} and using \eqref{mart}, it's not difficult to find  that 
\[Z_t^T=M_t^T+\int_0^t \lambda^T_s ds=M_t^T+\mu t+\int_0^t \int_0^s \varphi^T(s-v)dZ_v^Tds.\]
Using \cite[Lemma 22]{MR3449317}, we have $\int_0^t \int_0^s \varphi^T(s-v)dZ_v^Tds=\int_0^t  \varphi^T(t-s)Z_s^Tds$. Accordingly,
\begin{align}\label{Zt}
Z_t^T=M_t^T+\mu t+\int_0^t  \varphi^T(t-s)Z_s^Tds.
\end{align}
Since $(M_t^T)_{t\ge0}$ is martingale, we see that $\E[M_t^T]=0$ for all $t\ge0$. Hence, for $t\ge0$, we have 
\begin{align}\label{EZt}
\E[Z_t^T]=\mu t+\int_0^t  \varphi^T(t-s)\E[Z_s^T]ds=\mu t+\int_0^t  a_T\varphi(t-s)\E[Z_s^T]ds.
\end{align}
Applying Lemma \ref{classic} and noting \eqref{PSI2} and  $\varphi^{\star0}(t-s)ds=\delta_t(ds)$, we  thus conclude that 
\begin{align}\label{expZ}
\E[Z_t^T]=\mu\sum_{n=0}^\infty\int_0^t (a_T)^n \varphi^{\star n}(t-s)s\  ds=\mu t+\mu\int_0^t s\Psi^T(t-s)\ ds.
\end{align}
Now letting $t=Tx$ in \eqref{expZ}, we obtain 
\begin{align*}
    \E[Z_{Tx}^{T}]&=\mu Tx+\mu \int_{0}^{Tx}s\Psi^T(Tx-s)ds\\
   & \stackrel{s=Tv}{=}\mu Tx+\mu T^2\int_{0}^{x}v\Psi^T(T(x-v))dv.
    \end{align*}
    Following from Lemma \ref{phito}, we have 
    \begin{align*}
         \clim \E[Z_{Tx}^{T}]/T^2&=\clim\Big[ \frac{\mu x}{T}+\mu \int_{0}^{x}v\Psi^T(T(x-v))dv \Big]\\
       &  =\clim \mu\int_{0}^{x}v\Psi^T(T(x-v))dv\\
       &=\frac{\mu e^{\frac{\lambda}{m}x}}{m}\int_0^x ve^{-\frac{\lambda}{m}v}dv,
    \end{align*}
    which completes the proof of point (i).
    
\vip
    
  For point (ii),  the second inequality is nothing but  concluded by the first one due to  $x\in[0,1]$. We thus  only need to prove the first inequality. Let's first recall \eqref{phipsi} that $\Tilde \Psi^T(x)=\sum_{n\ge 1}(\Tilde \varphi^T)^{*n}(x)$  and \eqref{norm} that $\int_{0}^{\infty}\Tilde \Psi^T(s)ds=\frac{1}{a_T-1}.$ Using  \eqref{expZ} again with $t=Tx$ for $0\le x\le 1,$ 
    \begin{align*}
    \frac{a_T-1}{T}\E[Z_{Tx}^{T}]& = \frac{\mu(a_T-1)}{T} \Big[Tx+\int_{0}^{Tx}se^{b_T(Tx-s)}e^{-b_T(Tx-s)}\Psi^T(Tx-s)ds\Big]\\
    &\le \frac{\mu(a_T-1)}{T}\Big[T+Te^{b_TTx}\int_{0}^{\infty}\Tilde \Psi^T(s)ds\Big] \quad\quad  \text{(Using  $0\le x\le 1$ and  \eqref{PSI})}\\
    & \le  T \frac{\mu(a_T-1)e^{b_TTx}}{T}\Big(1+\frac{1}{a_T-1}\Big)= \mu  a_Te^{b_TTx},
    \end{align*}
    which concludes the proof.  
    
\end{proof}

\section{Proof of the asymptotic behavior of the process}
In this section, we are going to prove the laws of large numbers and the  central limit theorem for the process.
Let's first  give  the
\begin{proof}[Proof of Theorem \ref{thm1}]
Note \eqref{Zt}, \eqref{EZt} and \eqref{PSI2} that $\Psi^T(t)=\sum_{n\ge 1}(\varphi^T)^{*n}(t)$ and apply Lemma \ref{classic}, we have 
\begin{align*}
   Z_{Tx}^{T}-\E[Z_{Tx}^{T}]=\sum_{n\ge0}(a_T)^n\int_{0}^{Tx}\varphi^{*n}(Tx-s)M^T_{s}ds=M^T_{Tx}+\int_{0}^{Tx}\Psi^{T}(Tx-s)M^T_{s}ds.
\end{align*}
Since $ \E[M^{T}_s M^{T}_t]= \E[Z^{T}_{s\land t}].$ Then, it is easy to conclude from Proposition \ref{estiZ}-(ii) that, for any $0\le s,t\le T,$
\begin{align*}
    \E[(M^T_{s\land t})^2]=\E[Z^T_{s\land t}]\le\frac{\mu Ta_T}{a_T-1}e^{b_T s\land t}.
\end{align*}
Also note  \eqref{PSI} that $\Psi^T(t)= e^{b_T t}   \Tilde \Psi^T(t)$, we observe that for any $x\in[0,1]$,
\begin{align*}
    \E\Big[\Big(\int_{0}^{Tx}\Psi^{T}(Tx-s)M^T_{s}ds\Big)^2\Big]=&\E\Big[\int_{0}^{Tx}\int_{0}^{Tx}\Psi^{T}(Tx-s)\Psi^{T}(Tx-t)M^T_{s}M^T_{t}dsdt\Big]\\
    =&\int_{0}^{Tx}\int_{0}^{Tx}\Psi^{T}(Tx-s)\Psi^{T}(Tx-t)\E[M^T_{s}M^T_{t}]dsdt\\
    =&\int_{0}^{Tx}\int_{0}^{Tx}\Psi^{T}(Tx-s)\Psi^{T}(Tx-t)\E[Z^{T}_{s\land t}]dsdt\\
    =&e^{2b_TTx}\int_{0}^{Tx}\int_{0}^{Tx}\Tilde \Psi^{T}(Tx-s)\Tilde \Psi^{T}(Tx-t)e^{-b_T(s+t)}\E[Z^{T}_{s\land t}]dsdt\\
    \le &e^{2b_TTx}\int_{0}^{Tx}\int_{0}^{Tx}\Tilde \Psi^{T}(Tx-s)\Tilde \Psi^{T}(Tx-t)e^{-b_T(s\land t)}\E[Z^{T}_{s\land t}]dsdt\\
    \le& \frac{\mu Ta_T}{a_T-1} e^{2b_TTx}\Big(\int_{0}^{\infty} \Tilde \Psi^{T}(s)ds\Big)^2=\mu Ta_Te^{2b_TTx}(a_T-1)^{-3} .
\end{align*}
Whence, letting $T\to\infty$,
\begin{align*}
   \frac{(a_T-1)^2}{T^2e^{2b_TTx}} \E\Big[\Big(Z_{Tx}^{T}-\E[Z_{Tx}^{T}]\Big)^2\Big]
   \le\frac{Ca_T}{T(a_T-1)}\to 0,
\end{align*}
where  $C$ is a positive constant.
\end{proof}

Next, we prove our second main result related to the central limit theorem.  The idea of the proof comes  from  \cite[Chapter 5.3]{MR3187500}. We  now give  the proof.

\begin{proof}[Proof of Theorem \ref{thin}] We will divide the proof into several steps. 
\vip
{\bf Step 1.} Let's  first rewrite the  cumulated intensity of the Hawkes process.
Recalling  $M_t^T=Z_{t}^{T}-\int_0^t \lambda^T_s ds$ and $\lambda^T_s$ introduced in \eqref{sssy}, we have
\begin{align*}
   \int^{Tt}_0\lambda_{s}^{T}ds&=\mu Tt+\int_{0}^{Tt}\int_{0}^{s}\varphi^T(s-m)dZ_{m}^{T}ds\\
   &=\mu Tt+\int_{0}^{Tt}\int_{0}^{s}\varphi^T(s-m)dM_{m}^{T}ds+\int_{0}^{Tt}\int_{0}^{s}\varphi^T(s-m)\lambda_{m}^{T}dmds.
\end{align*}
Rearranging the above equality, we  then get
$$
\int^{Tt}_0\lambda_{s}^{T}ds-\int_{0}^{Tt}\int_{0}^{s}\varphi^T(s-m)\lambda_{m}^{T}dmds=\mu Tt+\int_{0}^{Tt}\int_{0}^{s}\varphi^T(s-m)dM_{m}^{T}ds.
$$
By Fubini's theorem and then changing the variable $u=s-m$, we have 
\begin{align}\label{lam}
    \int^{Tt}_0\lambda_{s}^{T}ds-\int_{0}^{Tt}\Big(\int_{0}^{Tt-u}\varphi^T(s)ds\Big)\lambda_{u}^{T}du=\mu Tt+\int_{0}^{Tt}\Big(\int_{0}^{Tt-u}\varphi^T(s)ds\Big)dM^T_u.
\end{align}
Define 
$$
\Phi^T(t):=\int_0^t \varphi^T(s) ds,\ \hbox{for}\  t\ge 0.
$$
Then, we have  $ \Phi^T(t)=0\ \hbox{for}\ t<0.$ Thus, the equality 
 (\ref{lam})  can be rewritten equivalently as
\begin{align*}
    \int^{Tt}_0\Big(1-\Phi^T(Tt-u)\Big)\lambda_{u}^{T}du=\mu Tt+\int_{0}^{Tt}\Phi^T(Tt-u)dM^T_u.
\end{align*}
By changing of variables of $u=Ts$,  we get
\begin{align*}
    \int^{t}_0T\Big(1-\Phi^T(T(t-s))\Big)\lambda_{Ts}^{T}ds=\mu Tt+\int_{0}^{t}\Phi^T(T(t-s))dM^T_{Ts}.
\end{align*}
 Putting $B^T_t=\frac{1}{\sqrt{T}}\int^{tT}_0\frac{dM^T_s}{\sqrt{\lambda^T_s}}$ and dividing $T$ on both sides, we can write
\begin{align}\label{convo}
 \int^{t}_0T\Big(1-\Phi^T(T(t-s))\Big)\frac {\lambda_{Ts}^{T}} {T} ds=\mu t+\int_{0}^{t}\Phi^T(T(t-s))\sqrt{\frac{\lambda_{Ts}^{T}}{T}}dB^T_{s}.
\end{align}

\vip

{\bf Step 2.} In this step, we verify that for the Skorohod topology,
\begin{align}\label{Brow}
    \clim ( B^T_{t})_{t\in  [0,1]}\stackrel{(d)}\longrightarrow  (B_t)_{t\in [0,1]},
\end{align}
where $(B_t)_{t\in [0,1]}$ is a Brownian motion.
To prove this, by Jacod-Shiryaev \cite[Theorem VIII-3.11]{MR1943877},
it suffices to verify that, as 
$T\to \infty$,
\vip

(a) The quadratic variation of $(B^T_t)_{t\in[0,1]}$, i.e. $[B^T,  B^T]_t \to t$ in probability, for all $t\in [0,1]$ fixed.

\vip

(b) $\sup_{t \in [0,1]}  |B^T_{t}-B^T_{t-}| \to 0$ in probability.
\vip
It is not difficult to check Point (b): for $t\in[0,1]$, use that the jumps of $M^T$ are always equal to $1$ (because the jumps of $M^T$ is counted by $Z^T$, which are all of size $1 $), and that $\lambda_{t}^{T}$ are always bigger than $\mu>0.$ Hence, we have  $\sup_{t \in [0,1]}  |B^T_{t}-B^T_{t-}|\le \frac{1}{\sqrt{\mu T}}$ which tends to $0$ as $T\to \infty.$ 

Concerning point (a), recall that The quadratic variation  $[M^T, M^T]_t=Z^T_t$ for $t\in[0,1]$.
Then for any $t\in[0,1]$,  we have
\begin{align*}
  [B^T,  B^T]_t = \frac{1}{T}\int^{tT}_0\frac{dZ^T_s}{\lambda^T_s}=\frac{1}{T}\int^{tT}_0\frac{dM^T_s}{\lambda^T_s}+\frac{1}{T}\int^{tT}_0\frac{\lambda^T_s}{\lambda^T_s}ds=\frac{1}{T}\int^{tT}_0\frac{dM^T_s}{\lambda^T_s}+t.
\end{align*}
Hence, to get point (a), it suffices to prove that 
$
\E\Big[\big(\frac{1}{T}\int^{tT}_0\frac{dM^T_s}{\lambda^T_s}\big)^2\Big]\to 0
$ as $T\to \infty$.
In fact, the Burkholder-Davis-Gundy inequality implies that for any $t\in[0,1]$
\begin{align*}
    \E\Big[\Big(\frac{1}{T}\int^{tT}_0\frac{dM^T_s}{\lambda^T_s}\Big)^2\Big]\le& \frac{C}{T^2} \E\Big[\int^{tT}_0\frac{dZ^T_s}{(\lambda^T_s)^2}\Big] \\
    =& \frac{C}{T^2} \E\Big[\int^{tT}_0\frac{dM^T_s}{(\lambda^T_s)^2}+\int^{tT}_0\frac{\lambda^T_sds}{(\lambda^T_s)^2}\Big]\\
    \le& \frac{C}{T^2} \E\Big[\int^{tT}_0\frac{dM^T_s}{(\lambda^T_s)^2}\Big]+\frac{CtT}{\mu  T^2}
  \le \frac{C}{\mu T} .
    \end{align*}
The last inequality follows from that $\E\Big[\int^{tT}_0\frac{dM^T_s}{(\lambda^T_s)^2}\Big]=0$. Now letting $T\to\infty$,  we end the step.

\vip

{\bf Step 3.} In the third step, we will verify that the sequence  $\frac{1}{T}\int^{t}_0\lambda_{Ts}^{T}ds$ is tight in $D[0,1]$ equipped with the Skorohod topology. It is obvious  that $\frac{1}{T}\int^{t}_0\lambda_{Ts}^{T}ds$ is continuous for $t\in[0,1]$. 
Recalling $Z^T$ and $\lambda^T$ introduced in \eqref{sssy}, we find 
$$\E[Z_{Tt}^T]=\E\left[\int_0^{Tt} \lambda^T_s ds\right] \stackrel{s=Tv}{=}\E\left[T\int_0^{t} \lambda^T_{Tv} dv\right].$$
Whence,  using  again Proposition \ref{estiZ}-(i) for any fixed $x\in[0,1]$, we conclude that there exists a constant $C>0$ such that for any $T$ and any $t\in[0,1]$, we have 
$$
\E\Big[\frac{1}{T}\int^{t}_0\lambda_{Ts}^{T}ds\Big]=\E\Big[\frac{Z^T_{Tt}}{T^2}\Big]\le \E\Big[\frac{Z^T_{T}}{T^2}\Big]\le C ,
$$
which finishes the step.
\vip
By Prokhorov's theorem,  there exists some increasing sequence  $\{T_i\}_{i\ge 1}$ satisfying  $\lim_{i\to\infty}T_i=\infty$ and a continuous process $\phi(t),$ such that     $\lim_{i\to\infty}\frac{1}{T_i}\int^{t}_0\lambda_{T_is}^{T_i}ds=\phi(t).$ 

\vip

{\bf Step 4.} In the last step, we are going to prove the result.
For any positive smooth function $K(\cdot)$ supported on $\R^+$,  taking the convolutions of the both sides of $(\ref{convo}),$ we get 
\begin{align}
    & \int^{t}_0\int^{s}_0K(t-s)T\Big(1-\Phi^T(T(s-s'))\Big)\frac {\lambda_{Ts'}^{T}} {T}ds' ds\label{5}\\
    &=\mu \int^t_0K(t-s)sds+ \int^{t}_0\int_{0}^{s}K(t-s)\Phi^T(T(s-s'))\sqrt{\frac{\lambda_{Ts'}^{T}}{T}} dB^T_{s'}ds\label{6}.
\end{align}
For $(\ref{5}),$  using Fubini's theorem and  taking the substitution $s=k+s'$, it is equal to 
\begin{align*}
    &\int^{t}_0\int^{t}_{s'}K(t-s)T\Big(1-\Phi^T(T(s-s'))\Big)\frac {\lambda_{Ts'}^{T}} {T}ds ds'\\
 &=\int^{t}_0\int^{t-s'}_0K(t-s'-k)T\Big(1-\Phi^T(Tk)\Big)\frac {\lambda_{Ts'}^{T}} {T}dk ds'\\
  & \stackrel{k'=Tk}{=}\int^{t}_0\int^{T(t-s')}_0 K\Big(t-s'-\frac{k'}{T}\Big)\Big(1-\Phi^T(k')\Big)\frac {\lambda_{Ts'}^{T}} {T}dk'ds'.
\end{align*}
Recalling $\Phi^T(t):=\int_0^t \varphi^T(s) ds$ and using  Fubini's theorem again, we can easily find that 
$$
\int_0^\infty (a_T-\Phi^T(k'))dk'=\cint\int_{k'}^\infty \varphi^T(s)ds dk'=a_T\int_0^\infty s\varphi(s) ds=m a_T.
$$
Hence, by dominated convergence theorem, we have 
\begin{align*}
   & \clim \int^{T(t-s')}_0 K\Big(t-s'-\frac{k'}{T}\Big)\Big(1-\Phi^T(k')\Big) dk'\\
    &= \clim \left[ \int^{T(t-s')}_0 K\Big(t-s'-\frac{k'}{T}\Big)\Big(a_T-\Phi^T(k')\Big) dk'+\int^{t-s'}_0 K\Big(t-s'-k\Big)T(1-a_T) dk \right]\\
    &= K(t-s')\cint\varphi(k') dk'-\lambda\int^{t-s'}_0 K\Big(t-s'-k\Big) dk\\
    &=m  K(t-s')-\lambda \int_0^{t-s'}K(u)du.
\end{align*}
Let's recall that  $\frac{1}{T_i}\int^{t}_0\lambda_{T_is'}^{T_i}ds'\to \phi_t$ as $i\to\infty$.  And using integration by parts and noticing $\phi_0=0,$  we have
\begin{align*}
    \int^{t}_0 \Big(\int_0^{t-s}K(u)du\Big)d\phi_s= \int^{t}_0 \phi_s K(t-s)ds.
\end{align*}
Hence,  taking  the  limit, \eqref{5} turns into 
\begin{align*}
  &  \lim_{i\to\infty} \int^{t}_0\int^{t-s'}_0K(t-s'-k)T_i\Big(1-\Phi^T_i(T_ik)\Big)\frac {\lambda_{T_is'}^{T_i}} {T_i}dk ds'\\
  &=m  \int_0^t K(t-s)d\phi_s-\lambda\int^{t}_0 K(t-s)\phi_sds. 
\end{align*}
Next, we are going to analyze $(\ref{6}).$ We observe that
\begin{align*}
    \clim \int_{0}^{t}K(t-s)\Phi^T(T(s-s'))ds= \clim \int_{s'}^{t}K(t-s)\Phi^T(T(s-s'))ds=\int_{s'}^{t}K(t-s)ds,
\end{align*}
since $\clim \Phi^T(T(s-s'))=1$ when $s\ge s'$. Hence

\begin{align*}
    \lim_{i\to\infty} \int^{t}_0\int_{0}^{s}K(t-s)\Phi^{T_i}(T_i(s-s'))\sqrt{\frac{\lambda_{T_is'}^{T_i}}{T_i}}dB^{T_i}_{s'}ds = \int_0^t  \Big(\int_0^{s}\sqrt{\frac{d\phi_{s'}}{ds'}}dB_{s'}\Big)K(t-s)ds.
\end{align*}
Since $(\ref{5})=(\ref{6}),$ we  thus  conclude that for any smooth function $K$,
\begin{align*}
    m  \int_0^t K(t-s)d\phi_s-\lambda\int^{t}_0 K(t-s)\phi_sds=\mu \int^t_0K(t-s)sds+\int_0^t  \Big(\int_0^{s}\sqrt{\frac{d\phi_{s'}}{ds'}}dB_{s'}\Big)K(t-s)ds,
\end{align*}
which implies  that $\phi_t$ satisfies 
\begin{align*}
    m d\phi_t=(\mu t+\lambda\phi_t)dt+\Big(\int_0^{t}\sqrt{\frac{d\phi_{s'}}{ds'}}dB_{s'}\Big)dt.
\end{align*}
Let $X_t=\frac{d\phi_{t}}{dt},$ we thus have $X_t$ satisfy
\begin{align}\label{sde}
     X_t=\frac{1}{m}\int_0^t\Big(\mu+\lambda X_s\Big)ds+\frac{1}{m}\int_0^t \sqrt{X_s}dB_s.
\end{align}
 Moreover,  due to
$\frac{\int_0^{t}\lambda_{Ts}^Tds}{T}=\frac{\int_0^{Tt}\lambda_{s}^Tds}{T^2}$, we find 
the difference 
\[ \frac{Z_{Tt}^T}{T^2}-\frac{\int_0^{t}\lambda_{Ts}^Tds}{T}=\frac{Z_{Tt}^T}{T^2}-\frac{\int_0^{Tt}\lambda_{s}^Tds}{T^2}=\frac{M^T_{Tt}}{T^2},\]
is a martingale. Thus, we have  for any $\epsilon>0$, 
$$\p\left(\sup_{t\in[0,1]}\left|\frac{M^T_{Tt}}{T^2}\right|\ge \epsilon\right)\le\frac{4}{T^4}\E\left[\int_0^T\lambda_{s}^Tds \right]\to0,$$
as $T\to\infty$. In fact, using the Doob's martingale inequality,
\begin{align*}
\E\left[\sup_{t\in[0,1]}\left|\frac{M^T_{Tt}}{T^2}\right|\right]\le \frac{4}{T^4}\E[(M_T^T)^2]=\frac{4}{T^4}\E[N_T^T]=\frac{4}{T^4}\E\left[\int_0^T\lambda_{s}^Tds \right].
\end{align*}
Consequently, for the Skorohod topology,
\begin{align*}
    \clim \frac{Z_{Tt}^T}{T^2}= \clim \frac{\int_0^{t}\lambda_{Ts}^{T}ds}{T}\stackrel{(d)}\longrightarrow\int^t_0X_sds,\hbox{ for} \  t\in[0,1].
\end{align*}
Finally, we can readily find that the SDE \eqref{sde} admits a unique strong solution by using \cite[Theorem 3.2]{MR1011252}. Indeed, the drift coefficient $b(x):=\frac\mu{m}+\frac\lambda{m}x$ is Lipschitz continuous and the diffusion coefficient $\sigma(x):=\frac{1}{m}\sqrt{x}$ is $1/2$-H\"{o}lder continuous on $[0,\infty)$. Together with the previous steps, we conclude the conclusion.
\end{proof}

\section*{Acknowledgements}
\quad~~
We would like to thank greatly Sylvain Delattre for fruitful discussions. And Liping Xu is supported by National Natural Science Foundation of China (12101028). 
\bibliographystyle{abbrv} 
\bibliography{hawkes}

\end{CJK}
\end{document}